\documentclass[11pt]{article}
\usepackage[a4paper]{anysize}\marginsize{3.5cm}{3.5cm}{1.3cm}{2cm}
\pdfpagewidth=\paperwidth \pdfpageheight=\paperheight
\usepackage{amsfonts,amssymb,amsthm,amsmath,eucal}
\usepackage{pgf}
\usepackage{bbm,array}
\usepackage{tikz} 
\usepackage{float}
\restylefloat{table}
\usepackage{subfigure}
\usepackage{caption}
\usetikzlibrary{arrows}
\usepackage{multicol}
\usepackage{multirow}

\theoremstyle{plain}
\newtheorem{thm}{Theorem}[section]
\newtheorem{theorem}[thm]{Theorem}

\newtheorem{lemma}[thm]{Lemma}

\newtheorem{corollary}[thm]{Corollary}

\theoremstyle{definition}
\newtheorem{definition}[thm]{Definition}
\newtheorem{remark}[thm]{Remark}

\newtheorem{fact}[thm]{Fact}

\newtheorem{thevarthm}[thm]{\varthmname}

\newenvironment{varthm*}[1]{\trivlist\item[]{\bf #1.}\it}{\endtrivlist}

\newcommand\be{\begin{eqnarray*}}
\newcommand\ee{\end{eqnarray*}}

\newcommand\newop[2]{\def#1{\mathop{\rm #2}\nolimits}}
\newop\mod{mod}
\newop\log{log}
\newop\ord{ord}
\newop\Gal{Gal}
\newop\SL{SL}
\newop\GL{GL}
\newop\Bl{Bl}
\newop\mult{mult}
\newop\mass{mass}
\newop\div{div}
\newop\codim{codim}
\newop\sing{sing}
\newop\vdim{vdim}
\newop\edim{edim}
\newop\Ass{Ass}
\newop\size{size}
\newop\reg{reg}
\newop\areg{areg}
\newop\asreg{asreg}
\newop\satdeg{satdeg}
\newop\supp{supp}
\newop\gin{gin}
\newop\ini{in}
\newop\vol{vol}
\newop\sat{sat}
\newop\length{length}
\newop\depth{depth}
\newop\characteristic{char}

\def\keywordname{{\bfseries Keywords}}%
\def\keywords#1{\par\addvspace\medskipamount{\rightskip=0pt plus1cm
\def\and{\ifhmode\unskip\nobreak\fi\ $\cdot$
}\noindent\keywordname\enspace\ignorespaces#1\par}}
\def\subclassname{{\bfseries Mathematics Subject Classification
(2000)}\enspace}
\def\subclass#1{\par\addvspace\medskipamount{\rightskip=0pt plus1cm
\def\and{\ifhmode\unskip\nobreak\fi\ $\cdot$
}\noindent\subclassname\ignorespaces#1\par}}

\definecolor{ttqqqq}{rgb}{0.,0.,0.}
\definecolor{zzttqq}{gray}{0.4}

\begin{document}

\author{ Magdalena ~Lampa-Baczy\'nska}
\title{The effect of points fattening on del Pezzo surfaces}
\date{\today}
\maketitle
\thispagestyle{empty}

\begin{abstract}

	In this paper, we study the fattening effect of points over the complex numbers for del Pezzo surfaces  $\mathbb{S}_r$ arising by
blowing-up of $\mathbb{P}^2$ at $r$ general points, with $ r \in \{1, \dots,  8 \}$.  Basic questions
when studying the problem of points fattening on an arbitrary variety are what is the minimal growth
of the initial sequence and how are the sets on which this minimal growth happens
 characterized geometrically. We provide complete answer for del Pezzo surfaces.

 \subclass{52C30,  14N20, 05B30}
\end{abstract}


\section{Introduction}\label{intro}

  In this paper, we follow the approach to fat point schemes 
initiated by Bocci and Chiantini in \cite{BocCha11}.
The  initial degree $\alpha(I)$ of a homogeneous
ideal $I \subset \mathbb{C}[\mathbb{P}^n]$ is the least degree $t$ such that the homogeneous component $I_t$
in degree $t$ is non-zero. Although this notion was known since $1981$ (see \cite{Chu81}),
Bocci and Chiantini  used this invariant for the first time in order to study fat points subschemes in the projective plane.

This definition  can be extended  to symbolic powers $I ^{(m)}$  of $I$, namely $\alpha(I^{(m)})$ is the least degree $t$ such that the homogeneous component $(I ^{(m)})_t$
in degree $t$ is non-zero. The sequence
$$\alpha(I), \alpha(I^{(2)}), \alpha(I^{(3)}), \dots$$
is called the \emph{initial sequence} of $I$.

Let $Z\subset \mathbb{P}^2 (\mathbb{C})$ be the set of points and $I$ be its radical ideal. By Nagata-Zariski theorem (\cite{EIS}, Theorem $3.14$) the ideal of scheme $mZ$ is the $m-$th symbolic power of $I$. Bocci and Chiantini proved, among others,
that sets of points $Z$ in $\mathbb{P}^2 (\mathbb{C})$ such that
$$\alpha(I^{(2)}) -\alpha(I)=1,$$
 are  either contained in a single line or form the so-called star-configuration.

Results of Bocci and Chiantini have been generalized in  \cite{DST13}
by Dumnicki, Szemberg and Tutaj-Gasi\'nska. They were studying
configurations of points  in $\mathbb{P}^2 (\mathbb{C})$ with
$$\alpha(I^{(m+1)})-\alpha(I^{(m)})=1$$
for some $m\geq 2$ and obtained their full characterization (see \cite{DST13}, Theorem $3.4$).

These considerations were extended for another types of spaces. Except for spaces $\mathbb{P}^{n}$  the problem of points fattening was considered among others by me in \cite{Lanckorona1} for the space $\mathbb{P}^1 \times \mathbb{P}^1$  and by Di Rocco, Lundman and Szemberg in \cite{DLS13} for Hirzebruch surfaces (with appropriately modified definition of the initial degree).

The aim of this paper is to make similar classification with respect to points fattening on del Pezzo surfaces. In this papper a del Pezzo surface (over an arbitrary field) is a smooth surface $X$ with the ample anticanonical bundle $-K_{X}$.

 In fact considerations on points fattening effect was initiated on del Pezzo surface  $\mathbb{P}^2(\mathbb{C})$ and this path of research was continued to another one, namely $\mathbb{P}^1 \times  \mathbb{P}^1$. Over $\mathbb{C}$ there are exactly 10 del Pezzo surfaces: $\mathbb{P}^2$, $\mathbb{P}^1 \times  \mathbb{P}^1$ and $8$ surfaces $\mathbb{S}_r$ arising by blowing-up of $\mathbb{P}^2$ in $r$ general points, where $1 \leq r \leq 8$. In this paper, we complete a picture for the last $8$ del Pezzo surfaces. More precisely, for each of the surfaces  $\mathbb{S}_r$ we establish maximal integer $m$, such that 
 $$\alpha(I)= \alpha(I^{(2)})= \cdots = \alpha(I^{(m)}) =1$$ holds and we describe subschemes realizing this sequence of equalities. We focus mainly on the smallest possible value of $\alpha(I)$, namely $1$. In the case of the surfaces $\mathbb{S}_1$ and $\mathbb{S}_2$ we additionally give characterization of subschemes satisfying a more general condition, namely $$\alpha(I^{(m)})= \alpha(I^{(m+1)}) \cdots = \alpha(I^{(m+a)})$$ for some integers $m$ and $a$. We conclude our paper presenting a Chudnovsky-type inequlity.
 
 \section{Basic notions and auxiliary facts}

 The original definition of the initial degree given in \cite{BocCha11} was extended in \cite{DLS13} for arbitrary smooth projective variety with an ample class.

\begin{definition}\textbf{(Initial degree)}
	\label{alfa}
	Let $X$  be a smooth projective variety with an ample line bundle $L$ on $X$ and let $Z$ be a reduced subscheme of $X$
	defined by the ideal sheaf ${\mathcal I}_Z \subset {\mathcal O}_{Z}$. For a positive integer $m$ the initial degree
	(with respect to $L$) of the subscheme $mZ$ is the integer
	$$\alpha(mZ)= \alpha({\mathcal I}^{(m)}_{Z}):=\min\left\{d:\; H^0(X,dL\otimes{\mathcal I}^{(m)}_{Z})\neq 0\;\right\}.$$
\end{definition} 
\noindent  Analogously, the initial sequence  (with respect to $L$)  of a subscheme $Z$ is the sequence
$$\alpha(Z), \alpha(2Z), \alpha(3Z), \dots$$ 

The initial sequence is sequence of positive integers with the following properties:

\begin{fact}
\label{basic}

\hspace{0,1cm}
\begin{itemize}
  \item[1)] The initial sequence is weakly growing, i.e., $\alpha(mZ)\leq \alpha(nZ)$ for $ n\geq m$.
  \item[2)] The initial sequence is subadditive, i.e., $\alpha((m+n)Z) \leq \alpha(mZ)+ \alpha(nZ)$.
  \item[3)] The initial sequence is monotonic with respect to the subcheme, i.e. if $Z \subset W$
  then $\alpha(mZ) \leq \alpha(mW)$.
\end{itemize}
\end{fact}
Properties in Fact \ref{basic} are generally known facts, thus we take them for granted.

The choice of a line bundle $L$ strictly depends on the variety $X$. 
In the projective plane the $\alpha$--invariant was taken with respect to the line bundle ${\mathcal O}_{\mathbb{P}^2}(1)$. This line bundle is  $\frac{1}{3}$ of
the anticanonical bundle $-K_{\mathbb{P}^2}={\mathcal O}_{\mathbb{P}^2}(3)$.
Similarly  on $\mathbb{P}^1\times \mathbb{P}^1$ it is natural to work with the $\alpha$--invariant taken with respect
to the line bundle ${\mathcal O}_{\mathbb{P}^1\times \mathbb{P}^1}(1,1)$. In this case the line bundle
is  half of the anti-canonical divisor as on $\mathbb{P}^1\times \mathbb{P}^1$ we have
$-{ K}_{\mathbb{P}^1\times \mathbb{P}^1}= {\mathcal O}_{\mathbb{P}^1\times \mathbb{P}^1}(2,2)$.

The most natural choice of the line bundle on del Pezzo surfaces $\mathbb{S}_{r}$ seems to be the ancticanonical bundle $$\mathbb{L}_{r}= -K_{\mathbb{S}_{r}}= 3H- E_1- \dots - E_{r},$$
which is not divisible in the Picard group $Pic(\mathbb{S}_{r})$.

The fattening effect can be also considered more generally for graded linear series.

Let $V_{\bullet} = {\bigoplus}_{d\geq 0} V_{d}  \subseteq {\bigoplus}_{d\geq 0} H^{0}(X, dL)$ be a graded linear system. We define
$$\alpha_{V_{\bullet}}(mZ)= \min \{ d: \exists s \in V_{d} : \mult_{Z}(s) \geq m \}.$$
Then we have the following property.

\begin{lemma}
\label{systems}
Let $V_{\bullet} \subseteq W_{\bullet}$ be graded linear systems. Then
$$\alpha_{V_{\bullet}}(mZ) \geq \alpha_{W_{\bullet}}(mZ).$$
\end{lemma}
\begin{proof}
It follows immediately from the fact, that $W_{\bullet}$ has more sections than $V_{\bullet}$.
\end{proof}

\begin{corollary}
Let $t \geq s$  and let $V_{\bullet} = \bigoplus V_{d}$ and $W_{\bullet} = \bigoplus W_{d}$, where
$$V_{d}=H^{0}(\mathbb{S}_{t}, 3dH- d \cdot \sum_{i=1}^{s} E_{i})\subseteq H^{0}(\mathbb{S}_{t},3dH),$$
$$W_{d}=H^{0}(\mathbb{S}_{t},3dH- d \cdot \sum_{i=1}^{t} E_{i})\subseteq H^{0}(\mathbb{S}_{t},3dH).$$
Then $W_{d} \subseteq V_{d}$ for all $d \geq 0$. Lemma \ref{systems} implies that
 $$\alpha_{W_{\bullet}}(mZ) \geq \alpha_{V_{\bullet}}(mZ).$$
 \end{corollary}
 
 \begin{remark}
 Note that if $Z \subseteq \mathbb{S}_{t}  \setminus \{E_{s +1}, \dots, E_{t} \}$, then $\alpha_{W_{\bullet}}(mZ) $
 is $\alpha(mZ)$ counted on the surface $\mathbb{S}_{t}$, whereas $\alpha_{V_{\bullet}}(mZ)$ can be regarded either as
 $\alpha(mZ)$ on $\mathbb{S}_{t}$ or $\alpha(mZ)$ on $\mathbb{S}_{s}$. This allows us to compare $\alpha$'s computed on
 various del Pezzo surfaces, provided that it makes sense to consider the underlying set $Z$ on both surfaces.
\end{remark}

The study of sequences of minimal growth can be paralleled by the investigation of
ratios of the type
$$\frac{\alpha(mZ)}{m}.$$
There are some estimates for such quotients. The  first result of this kind, 
$$\frac{\alpha(mZ)}{m} \geq \frac{\alpha(Z)+1}{2}$$ appears in a work of Chudnovsky and it concerns
finite sets of points $Z$ in the projective plane (see \cite{HarHun13}, Proposition $3.1$ and \cite{EV83}). It was generalized to other spaces, i.e. $\mathbb{P}^3$, $\mathbb{P}^1 \times \mathbb{P}^1$ and Hirzebruch surfaces (see \cite{DLS13}, \cite{DST13}, \cite{Jan13} and \cite{Lanckorona1}). All these inequalities are called by the common name \textit{Chudnowsky-type results}. In Section $4$ we present our result of this kind for del Pezzo surfaces $\mathbb{S}_{r}$ with $r \leq 6$. 

Our paper is concluded by a comparison of points fattening effect for surface $\mathbb{S}_{1}$ considered as a del Pezzo surface and on the other hand as a Hirzebruch surface.

\section{The points fattening effect on $\mathbb{S}_{r}$}

In this section we present some results concerning the  fattening effect on $\mathbb{S}_r$.  Let us recall that $\mathbb{S}_r$ arises as the blowing-up of the complex
	projective plane in fixed $r$ general points  $P_1 ,..., P_r$. We denote by $f_{r} : \mathbb{S}_r \rightarrow \mathbb{P}^2$ the blow up, where $E_1, \dots, E_r$ are the exceptional divisors. If $r$ is fixed,  then we write simply $f$ instead of $f_r$.
	
In further considerations we will use the following observations about blow-ups.

\begin{remark} \label{adapted}
	\label{Ecomponent}
	If $F$ is a plane curve of degree $3k$  in $\mathbb{P}^2$ passing through the points $P_1, \dots, P_{r}$, so that $\mult_{P_{i}}(F)=m_{i} \geq k$ for $i \in \{1, \dots, r \}$,
	then $E_{i}$ is a $(m_{i}-k)-$tuple component of the divisor $f^{*}_{r}(F)- kE_1 - \dots - k E_{r}$ in the system $$|3kH- kE_1 - \dots - kE_{r}|= |-3k K_{\mathbb{S}_r}|.$$
\end{remark}

\begin{definition}\textbf{(Adapted transform)}
	We keep the notation as in Remark \ref{adapted}. The adapted transform of $F$ is the divisor
	$$A(F):=f^{*}_{r}(F)- kE_1 - \dots - k E_{r}=\widetilde{F} + {\sum}_{i=1}^{r} (m_{i} - k)E_{i},$$
	where $\widetilde{F}$ denotes the proper transform of $F$.
\end{definition}

\begin{lemma}
	\label{estimation}
	Let $D \in |k (3H - E_1 - \cdots - E_r)|$ for fixed $1 \leq r \leq 8$ and let $Q \in \mathbb{S}_{r}$. Then
	\begin{equation}
	\label{onE}
	{\mult}_{Q}(D) \leq 2 \cdot {\mult}_{f_{r}(Q)}(f_{r}(D))-k ,
	\end{equation}
	if  $Q \in E_1  \cup \ldots \cup  E_{r}$ and
	\begin{equation}
	\label{pozaE}
	{\mult}_{Q}(D) ={\mult}_{f_{r}(Q)}(f_{r}(D)) \leq 3k ,
	\end{equation}
if $Q \notin E_1  \cup \ldots \cup  E_{r}$. Furthermore, if equality holds in (\ref{pozaE}), then $f_{r}(D)$ is a union of lines through $f_{r}(Q)$.
\end{lemma}

\begin{proof}
	Let $D \in |k (3H - E_1 - \cdots - E_r)|$  and $Q \in \mathbb{S}_{r}$. Then $\deg(f_{r}(D))=3k$. Let us denote by $m= {\mult}_{Q}(D)$.
	
	First we consider the situation, when $Q \notin E_1  \cup \ldots \cup  E_{r}$. Since $f_{r}$ is an isomorphism except the points $\{P_1, \dots, P_r \}$,
	then ${\mult}_{f_{r}(Q)}(f_{r}(D))=m$. The multiplicity of the singular
	point of the plane curve can be at most the degree of this curve, thus  $f_{r}(D)$ may have at most $3k-$tuple points, what finishes
	the proof of statement (\ref{pozaE}).
	
	We assume now, that $Q \in E_i$ for some $i \in \{1, \dots, r \}.$ Let us denote by $F= f_{r}(D)$. Then
	$${\mult}_{Q}(D) = {\mult}_{P_{i}}(F)-k + {\mult}_{Q}(\widetilde{F}).$$
	But ${\mult}_{Q}(\widetilde{F}) \leq {\mult}_{P_{i}}(F)$, thus we finally obtain the statement (\ref{onE}).
\end{proof}

A natural consequence of Lemma \ref{estimation} is the following property for surfaces $\mathbb{S}_r$.
\begin{corollary} \label{onexceptional}
	If $Z \subset \mathbb{S}_r$  for $1 \leq r \leq 8$ satisfies the condition
	$$\alpha(mZ)= \dots =\alpha((m+t)Z)$$
	for some positive integers $m$ and $t \geq 3$, then $Z \subset E_1 \cup \dots 
	\cup  E_r$.
\end{corollary}

Now we turn to the main subject of this paper, namely a characterization of subschemes $Z$ with $$\alpha(Z)= \alpha(2Z) \cdots = \alpha(mZ)=1.$$ We begin with surfaces $\mathbb{S}_1$ and $\mathbb{S}_2$.

\subsection{Surfaces $\mathbb{S}_1$ and $\mathbb{S}_2$}

\begin{theorem} \label{s1.5jumps}
	Let $Z \subset \mathbb{S}_1$ be a finite set of points. Then the following conditions are equivalent
	\begin{itemize}
		\item[$i)$] $Z= \{ Q\} \subset E_1$,
		\item[$ii)$] $\alpha(Z)= \alpha(2Z)=\alpha(3Z)= \alpha(4Z)=\alpha(5Z)=1$.
	\end{itemize}
\end{theorem}

\begin{proof}
	The implication from $i)$  to $ii)$ is obvious. It is enough to consider the nonreduced curve $F=3L \subset \mathbb{P}^2$ for some line passing through the point $P_1$. Indeed, it gives rise to
	$$A(F)=f^{*}F-E_1= 3\widetilde{L} + 2 E_1$$
	in $\mathbb{S}_1$, which vanishes to order $5$ along $Q \in \widetilde{L} \cap E_1$.

	In order to prove the reverse implication, let $Z=\{Q_{1}, \dots , Q_{s} \}$ and
	we assume that $D \in |3H-E_1|$ is a divisor satisfying ${\mult}_{Q_{i}} (D) \geq 5$ for all points $Q_{i} \in Z$. By Corollary \ref{onexceptional} we have  that $Z\subset E_1$.

Let us consider possible types of cubic curves in the projective plane and their adapted transforms. The curve $F$ has to pass through the point $P_1$ 
	 and in order to get the highest possible multiplicities along the exceptional
	divisor $E_1$ it should have the highest possible multiplicity at $P_1$. We have the following types of cubic curves in $\mathbb{P}^2$:
	\begin{itemize}
		\item[a)] an irreducible cubic (possibly singular);
		\item[b)] a union of an irreducible conic and a line;
		\item[c)] a union of three lines (possibly not distinct).
	\end{itemize}

	\noindent
	In the case $a)$ the divisor $A(F)$ on $\mathbb{S}_1$ has points of multiplicity at most two.
	In the case b) the highest possible multiplicity of a point on $E_1$ is three, this happens in the case
	when the line is tangent to the conic at point $P_1$.

	Let us pass to the case $c)$. We know that the adapted transform of a curve $F$
consisting of some triple line $L$ has a quintuple point. Except this arrangement of three lines we never get the quintuple points, which completes the proof.
\end{proof}

\begin{remark}\label{bound}
	In fact one can weaken condition $ii)$ in Theorem \ref{s1.5jumps}. Assuming 
	$$\alpha(mZ)=\alpha((m+1)Z)= \alpha((m+2)Z)= \alpha((m+3)Z)= \alpha((m+4)Z)$$
	for some $m \geq 1$, implies that $Z=\{Q\} \subset \mathbb{S}_1$.
\end{remark}
\begin{proof}
Let $Z= \{Q_1, \dots, Q_{t} \}$  be  such that  $\alpha(mZ)= \dots = \alpha((m+4)Z)=k$ for some integers $k$ and $t$,
and let $D \in |3kH- kE_1|$ be a divisor such that
${\mult}_{Q_{i}}(D) \geq m+4$ for any point $Q_{i} \in Z$. Firstly, by Lemma \ref{estimation} we conclude that $Z\subset E_1$. 

Let us denote by $F=f(D)$ having $\deg(F)= 3k$.
 Since $F$ is of degree $3k$ its multiplicity at $P_1$ is at most $3k$. Hence the multiplicity of $E_1$ in $D$ is at most $2k$. This contributes to the multiplicity of $D$ at every point $Q_1, \dots, Q_{t}$. The remaining multiplicity at these points must come from components of $F$ passing through $P_1$ at directions corresponding to  $Q_1, \dots, Q_{t}$. We have

\begin{equation}
\label{estim1}
t(m+4) \leq \sum_{i=1}^{t} mult_{Q_{i}} D \leq 3k+2kt.
\end{equation}
On the other hand, since $\alpha (mZ)=k$, it must be

\begin{equation}
\label{estim2}
3(k-1)+2(k-1)t <t \cdot m,
\end{equation}
since otherwise one could find $3(k-1)$ lines through $P_1$. Their images in  $\mathbb{P}^2$ would show $\alpha (mZ) \leq k-1$ contradicting the assumption. Combining (\ref{estim1}) and (\ref{estim2}) we get that
$$3k-3+2kt-2t+4t< 3k+2kt$$
and thus $t< \frac{3}{2}$, what finally means that $Z$ is a single point.

\end{proof}

On $\mathbb{S}_1$, there also exist infinitely many sets  satisfying a weaker condition, namely
$$\alpha(mZ)= \dots =\alpha((m+3)Z),$$
and these sets are not necessarily the same as in Theorem \ref{s1.5jumps}.

\begin{theorem} \label{s1.4jumps}
	Let $Z \subset \mathbb{S}_1$ be a finite set of points and let $m$ be a positive integer. Then the following conditions are equivalent:
	\begin{itemize}
		\item[$i)$] $\alpha(mZ)= \dots =\alpha((m+3)Z)$;
		\item[$ii)$] $Z= \{ Q\} \subset E_1$  or $Z=\{Q_1, Q_2 \} \subset E_1$, where $Q_1 \neq Q_2$.
	\end{itemize}
\end{theorem}
\begin{proof}
	The sets in $ii)$ satisfy the condition
$$\alpha(mZ)= \dots =\alpha((m+3)Z),$$ for example with $m=1$ and $m=4$ respectively.
We will prove the opposite implication. Suppose now that $Z= \{Q_1, \dots, Q_{t} \}$  is a set such that  $\alpha(mZ)= \dots = \alpha((m+3)Z)=k$ for some integers $k$ and $t$ and let $D \in |3kH- kE_1|$ be a divisor such that
${\mult}_{Q_{i}}(D) \geq m+3$ for any point $Q_{i} \in Z$. Let us denote by $F=f(D)$, with $\deg(F)= 3k$. 

In fact we can repeat reasoning used in the proof of Remark \ref{bound}, but this time with the following estimates

\begin{equation}
\label{estim3}
t(m+3) \leq \sum_{i=1}^{t} mult_{Q_{i}} D \leq 3k+2kt,
\end{equation}

\begin{equation}
\label{estim4}
3(k-1)+2(k-1)t <t \cdot m.
\end{equation}
By (\ref{estim3}) combined with (\ref{estim4}) we get
$$3k-3+2kt-2t+3t< 3k+2kt,$$
what gives $t<3$.
\end{proof}

\begin{corollary}
For a finite set of points $Z \subset \mathbb{S}_1$  and a positive integer $m$ we have $$\alpha(mZ) < \alpha((m+5)Z).$$
\end{corollary}

\begin{proof}
    Suppose to the contrary that $Z= \{Q_1, \dots, Q_{t} \}$  is  such that  $\alpha(mZ)= \dots = \alpha((m+5)Z)=k$ for some integers $k$ and $t$, and let $D \in |3kH- kE_1|$ be a divisor such that
${\mult}_{Q_{i}}(D) \geq m+3$ for  $Q_{i} \in Z$. Let  $F=f(D)$, with $\deg(F)= 3k$. In the spirit of the proof of Remark \ref{bound}, we get estimates

\begin{equation}
\label{estim5}
t(m+5) \leq \sum_{i=1}^{t} mult_{Q_{i}} D \leq 3k+2kt,
\end{equation}

\begin{equation}
\label{estim6}
3(k-1)+2(k-1)t <t \cdot m.
\end{equation}
Combining (\ref{estim5}) and (\ref{estim6}) we obtain $t<1$, but $t$ is a positive integer, a contradiction.
\end{proof}

One may also consider sets with three  initial values equal to $1$. The list of possible types gets much longer but the arguments used to obtain their classification are similar to those used above. We refer the interested reader to \cite{thesis} (Subchapter $6.1$).

In further considerations lines joining points $P_1, \dots, P_r$ plsy an important role. From now on  we denote by $L_{ij}$ the line passing through the points $P_i$ and $P_j$ for fixed distinct $i,j \in \{1, \dots, r \}$. We are ready  to formulate analogous  results concerning the fattening effect on  $\mathbb{S}_2$.

\begin{theorem}
\label{5jumps}
The following conditions are equivalent:
\begin{itemize}
  \item[i)] $Z \subset \widetilde{L_{12}} \cap (E_1 \cup E_2);$
  \item[ii)] $\alpha(Z)= \dots =\alpha(5Z)=1.$
\end{itemize}
\end{theorem}

\begin{theorem}
\label{4jumps}
The following conditions are equivalent:
\begin{itemize}
  \item[i)] $Z= \{Q\}  \subset  (E_1 \cup E_2) \setminus \widetilde{L_{12}};$
  \item[ii)] $\alpha(Z)= \dots =\alpha(4Z)=1$ and $\alpha(5Z)>1.$
\end{itemize}
\end{theorem}

\begin{theorem}
\label{5boundary}
For any finite set of points $Z \subset \mathbb{S}_2$ and any positive integer $m$ we have $\alpha(mZ) < \alpha((m+5)Z)$.
\end{theorem}

\begin{theorem}
\label{5ingeneral}
For any finite set of point $Z \subset \mathbb{S}_2$ the following conditions are equivalent:
\begin{itemize}
  \item[i)] there exists a positive integer $m$ such that $\alpha(mZ)= \dots =\alpha((m+4)Z);$
  \item[ii)] $Z \subset (E_1 \cup E_2) \cap \widetilde{L_{12}}$.
\end{itemize}
\end{theorem}

\begin{theorem}
\label{4ingeneral}
Let $Z \subset \mathbb{S}_2$ be a finite set of points such that $Z \subset E_{i}$ for $i \in \{1,2 \}$. Then $Z$ satisfies the condition
\begin{equation}
\label{4.equality}
\alpha(mZ)= \dots = \alpha((m+3)Z)< \alpha((m+4)Z),
\end{equation}
for some positive integer $m$ if and only if $Z$ has the following form: either

\begin{itemize}
\item[a)] $Z= \{ Q \} \subset E_{i} \setminus \widetilde{L_{12}}$ for $i \in \{ 1,2 \}$, or
\item[b)] $Z= \{ Q,Q' \} \subset E_{i}$ for $i \in \{ 1,2 \}$, where $Q' \in \widetilde{L_{12}}$.
\end{itemize}
\end{theorem}

All theorems from \ref{5jumps} to \ref{4ingeneral} can be proved analogously as in the case of surface $\mathbb{S}_1$, or the reader can find alternative proofs in \cite{thesis}. In the subchapter $2.4$ of \cite{thesis} reader can also find a description of sets satisfying the condition $$\alpha(Z) =\alpha(2Z) =\alpha(3Z) =1.$$

Question about the maximal integer  $a$ satisfying condition
$$\alpha(mZ)= \alpha((m+1)Z) \cdots = \alpha((m+a)Z)$$ is not trivial and it is still an open problem for some of studied before surfaces. For example, fattening effect for the Hirzebruch surfaces is described in \cite{DLS13} only with respect to the condition 
$$\alpha(Z) = \dots =\alpha(mZ) =1.$$ The  available tools are not useful in the case of greater $r$. In particular, when considered sets have some points on the exceptional divisors. From that reason for remaining surfaces $\mathbb{S}_r$ we only establish maximal $m$, for which $\alpha(Z) = \dots =\alpha(mZ) =1$ hold and we describe sets $Z$ with that property.

\subsection{Surfaces $\mathbb{S}_r$ for $r \geq 3$}

The natural sequence of inclusions between linear systems 
$$|\mathbb{L}_{1}| \supset  |\mathbb{L}_{2}| \supset |\mathbb{L}_{3}| \supset  |\mathbb{L}_{4}| \supset  |\mathbb{L}_{5}| \supset  |\mathbb{L}_{6}| \supset  |\mathbb{L}_{7}| \supset  |\mathbb{L}_{8}|$$
suggests, that the sequence of equalities
$$\alpha(Z)= \alpha(2Z)= \dots =1$$
should become shorter with $r$ growing. In the case of $r=1$ and $r=2$ we had $\alpha(5Z)=1$ and  we proved moreover, that there it is not possible to obtain more than five consecutive initial values equal. 

We present now such a characterization for remaining surfaces $\mathbb{S}_r$.

\begin{theorem}
\label{s3.4jumps}
Let $Z\subset \mathbb{S}_3$ be a finite set of points. The following conditions are equivalent:
\begin{itemize}
  \item[$i)$] $\alpha(Z) = \dots = \alpha(4Z)=1$;
  \item[$ii)$] $Z= \{Q \}= E_{i} \cap \widetilde{L_{ij}}$ for distinct $i, j \in \{1,2,3 \}.$
\end{itemize}
\end{theorem}

\begin{theorem}
\label{s4.3jumps}
Let $Z \subset \mathbb{S}_4$ be a finite set of points. Then $Z$ satisfies equality $\alpha(Z)=\alpha(2Z)=\alpha(3Z)=1$, if and only if
it is one of the following sets:
\begin{itemize}
  \item[a)] $Z= \{Q \} \subset \widetilde{L_{ij}}$ for distinct $i,j \in \{1,2,3,4 \}$,
  \item[b)] $Z \subset \{Q, Q_1, Q_2\} \subset \widetilde{L_{ij}}$, where $Q_1 \in E_{i}$, $Q_2 \in E_{j}$ and
  $Q = \widetilde{L_{ij}} \cap \widetilde{L_{kl}}$ for pairwise distinct $i,j,k,l \in \{1,2,3,4 \}$,
  \item[c)] $Z \subset E_{i} \cap (\widetilde{L_{ij}} \cup \widetilde{L_{il}} \cup \widetilde{L_{ik}})$
  for pairwise distinct $i,j,k,l \in \{1,2,3,4 \}$,
  \item[d)] $Z= \{Q \} \subset E_{i}$ and $Q \in \widetilde{C} \cap \widetilde{L} $, where $C$ is an  irreducible conic curve passing
  through the points   $P_1, P_2, P_3, P_4$ and $L$ is the line tangent to $C$ at the point $P_{i}$ for $i \in \{1,2,3,4 \}$.
\end{itemize}
\end{theorem}

\begin{theorem}
Let $Z \subset \mathbb{S}_5$ be a finite set of points. Then the condition $\alpha(Z)= \alpha(2Z)= \alpha(3Z)=1$ is fulfilled if and only if
$Z= \{Q\}$ and the point $Q$ satisfies one of the following two conditions:
\begin{itemize}
\item[a)] $Q \in E_{i} \cap \widetilde{L_{i}} \cap \widetilde{C}$ for $i \in \{1,2,3,4,5 \}$, where $C$ is a conic passing through the points
$P_1, \dots, P_5$ and $L_{i}$ is a line tangent to $C$ at the point $P_{i}$,
\item[b)] $Q \in \widetilde{L_{ij}} \cap \widetilde{L_{kl}}$ for pairwise distinct $i,j,k,l \in \{1,2,3,4,5 \}$.
\end{itemize}
\end{theorem}

\begin{theorem}
\label{s6.3jumps}
Let $Z \subset \mathbb{S}_6$ be a finite set of points. Then $Z$ satisfies equality $\alpha(Z)= \alpha(2Z)= \alpha(3Z)=1$, if and only if
$Z= \{Q\}$ and the point $Q$ fulfils one of the following two conditions:
\begin{itemize}
\item[a)] $Q \in \widetilde{L_{ij}} \cap \widetilde{L_{kl}} \cap \widetilde{L_{mn}}$ for pairwise distinct $i,j,k,l,m,n \in \{1,2,3,4,5,6 \}$,
\item[b)] $Q \in E_{i} \cap \widetilde{L_{ij}} \cap \widetilde{C_{j}}$ for distinct $i,j \in \{1,2,3,4,5,6 \}$, where $C_{j}$ is a conic curve
determined by five points of $P_1, \dots, P_6$ excluding $P_{j}$ and $L_{ij}$ is of course the line passing through the points $P_{i}$ and $P_{j}$,
but  simultaneously $L_{ij}$ is the tangent line to the curve $C_{j}$ at the point $P_{i}$.
\end{itemize}
\end{theorem}

Let us note that $\mathbb{S}_6$ is the first example of surfaces $\mathbb{S}_{r}$, where the existence of a set $Z$ satisfying the condition 
$$\alpha(Z)= \alpha(2Z)= \alpha(3Z)=1$$
depends on the geometry of points $P_1, \dots, P_6$.
For given six points in general position in the projective plane, there always exists a cubic curve consisting of three lines, passing through these points.
But these lines do not have to intersect at one point (see Figure \ref{3lines}). It is a rather strong  requirement.

\begin{figure}[h]
	\centering
	\begin{minipage}[b]{0.47\linewidth}
	
		\centering
	\begin{tikzpicture}[line cap=round,line join=round,>=triangle 45,x=1.0cm,y=1.0cm]
\clip(3.805140000000000002,-4.07200000000000015) rectangle (10.040000000000004,4.04);
\draw [dash pattern=on 3pt off 3pt,domain=4.7140000000000002:7.5040000000000004] plot(\x,{(--9.7664-1.58*\x)/0.6799999999999997});
\draw [dash pattern=on 3pt off 3pt,domain=4.5140000000000002:9.340000000000004] plot(\x,{(--16.0352-2.16*\x)/-2.12});
\draw [dash pattern=on 3pt off 3pt,domain=5.7140000000000002:7.5040000000000004] plot(\x,{(-14.059199999999997--2.06*\x)/0.6199999999999992});
\begin{scriptsize}
\draw [fill=black] (6.56,-0.88) circle (2.5pt);
\draw [fill=black] (5.88,0.7) circle (1.5pt);
\draw[color=black] (6.060000000000002,0.90399999999999996) node {$P_1$};
\draw [fill=black] (8.68,1.28) circle (1.5pt);
\draw[color=black] (8.60860000000000003,1.6199999999999997) node {$P_2$};
\draw [fill=black] (5.94,-2.94) circle (1.5pt);
\draw[color=black] (6.220000000000002,-2.800000000000001) node {$P_3$};
\draw [fill=black] (4.875147405087775,-2.5966422665143423) circle (1.5pt);
\draw[color=black] (4.8060000000000002,-2.2600000000000007) node {$P_4$};
\draw [fill=black] (5.037285385967285,2.658072191428958) circle (1.5pt);
\draw[color=black] (5.2200000000000015,2.9) node {$P_5$};
\draw [fill=black] (6.890684528954191,0.21872601555747684) circle (1.5pt);
\draw[color=black] (7.2080000000000002,0.5599999999999996) node {$P_6$};
\end{scriptsize}
\end{tikzpicture}
		\caption{}
		\label{3lines}
	\end{minipage}
	\quad
	\begin{minipage}[b]{0.47\linewidth}
	\begin{tikzpicture}[line cap=round,line join=round,>=triangle 45,x=1.0cm,y=1.0cm]
\clip(3.000000000000001,-6.080000000000001) rectangle (12.620000000000005,2.6800000000000006);
\draw [rotate around={156.8111748614585:(6.447293946929891,-1.200127006811614)},dash pattern=on 3pt off 3pt] (6.447293946929891,-1.200127006811614) ellipse (3.012213514665936cm and 1.7591812283660395cm);
\draw [dash pattern=on 3pt off 3pt,domain=8.5000000000000001:10.0620000000000005] plot(\x,{(--24.798954390524234-2.5180395470136987*\x)/-0.6546004783349169});
\begin{scriptsize}
\draw [fill=black] (8.3,-0.28) circle (1.5pt);
\draw[color=black] (8.480000000000002,0.06) node {$P_1$};
\draw [fill=black] (9.28,-2.18) circle (2.5pt);
\draw[color=black] (9.70000000000003,-2.19600000000000005) node {$P_2$};
\draw [fill=black] (5.96,-2.9) circle (1.5pt);
\draw[color=black] (6.1400000000000015,-2.5600000000000005) node {$P_3$};
\draw [fill=black] (3.9537693477677376,0.3656249440815964) circle (1.5pt);
\draw[color=black] (3.9140000000000001,0.7000000000000002) node {$P_4$};
\draw [fill=black] (6.966233672642103,0.48765336456175185) circle (1.5pt);
\draw[color=black] (7.140000000000002,0.8200000000000001) node {$P_6$};
\draw [fill=black] (8.685969051542768,-4.4719503139512575) circle (1.5pt);
\draw[color=black] (8.960000000000003,-4.31400000000000015) node {$P_5$};
\end{scriptsize}
\end{tikzpicture}
		\caption{}
		\label{conic}
	\end{minipage}
\end{figure}

Similarly when the cubic splits into a conic and a line.
Each five points determine a conic curve in a unique way. But the line joining the sixth point with one of the previous five does not necessary need to
be tangent to this conic (see Figure \ref{conic}). It is also a situation, which may happen or not, and it depends of the arrangement of the starting six points (although
they are always in general position). It is a quite interesting phenomenon. Especially that for remaining two surfaces $\mathbb{S}_r$ the condition 
$$\alpha(Z)= \alpha(2Z)= \alpha(3Z)=1$$ in never satisfied.

\begin{theorem}
\label{s7.2jumps}
Let $Z \subset \mathbb{S}_7$ be a finite set of points. The equality $\alpha(Z)= \alpha(2Z)=1$ holds if and only if $Z$ is one of the following sets:
\begin{itemize}
\item[a)] $Z \subset \{Q_1, Q_2 \} \subset E_{i} \cap \widetilde{F}$, where $F$ is an irreducible singular cubic curve with the
singularity at the point $P_{i}$ for some  $i \in \{1, \dots, 7 \}$,
\item[b)] $Z=\{Q \}$, where $f(Q)$ is the double point of a singular cubic passing through the points  $P_1, \dots, P_7$ and
$Q \notin E_1 \cup \dots \cup E_7$,
\item[c)] $Z \subset \widetilde{L_{ij}} \cap \widetilde{C_{ij}}$ for distinct $i,j \in \{1, \dots, 7 \}$, where $C_{ij}$ is the irreducible conic
passing through the five points of $P_1, \dots, P_7$, distinct from $P_{i}$ and $P_{j}$.
\end{itemize}
\end{theorem}

Proofs of Theorems \ref{s3.4jumps} to \ref{s7.2jumps} are based on a review of plane cubics passing through the points $P_1, \dots , P_r$ (analogously as in the proof of Theorem \ref{s1.5jumps}). Thus we skip details here. We refer the curious reader to \cite{thesis}.

The line bundle $$\mathbb{L}_{8}= 3H- E_1 - \dots - E_8$$
has the least number of sections of all line bundles  $\mathbb{L}_{r}$ considered so far, namely $h^0 (\mathbb{L}_8) =1$.
For that reason we expected, that
 $$\alpha(2Z) \geq 2$$ here. Thus the fact, that there exists a set $Z$, where $\alpha(Z)= \alpha(2Z) =1$ was surprising.

\begin{theorem}
If $Z \subset \mathbb{S}_8$, then the equality $\alpha(Z)= \alpha(2Z) =1$ holds iff $Z= \{Q\}$, where $f(Q)$ is the singular point of an
irreducible cubic curve $F$ passing through points $P_1, \dots, P_8$ and $f(Q)$ is distinct of any point $P_{i}$ for $i \in \{1, \dots, 8\}$.
\end{theorem}

\begin{proof}

If $Z= \{Q\}$ and $f(Q)$ is double point on a cubic, distinct of any $P_{i}$, then it is obvious, that $Q$ has multiplicity $2$ on $\mathbb{S}_8$ and
of course $\alpha(Z)= \alpha(2Z) =1$.

Let us focus on the opposite implication. Let $D\in |\mathbb{L}_{8}|$ be such that
${\mult}_{Q} (D) \geq 2$ for any point $Q \in Z$.
The curve $F=f(D)$ has degree $3$ and passes through eight distinct points $P_1, \dots, P_8$ in general position. Then $F$ is irreducible. Irreducible
 cubic has at most one singular point and it can not be any of $P_{i}'$s (general points). Thus it must be $f(Q)$.

 To finish the proof we need to show, that there always exists a singular
 cubic curve  passing through $8$ given general points.

 Let us notice that cubics passing through  $8$ fixed points form a pencil, if no four points lie on a line and no seven lie on a conic. Since
  $P_1, \dots, P_8$ are in general position, the family of cubics passing through these points is a pencil. We denote it by $\mathcal{S}$.
  Every two cubics in $\mathcal{S}$ meet in nine points, thus the set $\{ P_1, \dots, P_8 \}$ determines a new point. This point is determined
  uniquely (Cayley–Bacharach theorem, see \cite{CaBa}, Theorem $1$). Let us denote it by $P_9$.

  Let  $\widetilde{\mathbb{S}}$ be a blow-up of $\mathbb{P}^2$ in all nine points $P_1, \dots, P_9$. Then  $\widetilde{\mathbb{S}}$ is the total space of
the pencil $\mathcal{S}$ and we have the morphism
$$\varphi : \widetilde{\mathbb{S}} \rightarrow \mathbb{P}^1 ,$$
whose fibers are the elements of $\mathcal{S}$. Let $e(\cdot)$ be the topological Euler characteristic. Thus we have
$$e(\mathcal{S})=e(\mathbb{P}^2) + 9 =12.$$

Suppose now, to the contrary, that $\varphi$ has only smooth fibers. Then from the topological point of view we have
$$\widetilde{\mathbb{S}} = \mathbb{P}^1 \times \mathcal{E},$$
where $\mathcal{E}$ is an elliptic curve. We have then
$$e(\mathcal{S})=e(\mathbb{P}^1 ) \cdot e(\mathcal{E})=2 \cdot 0=0.$$
Thus $\mathcal{S}$ must contain  singular fibers. Since the points $P_1, \dots, P_8$ are in general position, these singular fibers
are irreducible cubics, which ends the proof.
\end{proof}

\section{The Chudovsky-type result for surfaces $\mathbb{S}_r$}
We conclude our considerations  by a lower bound on the growth rate of the initial sequence for surfaces $\mathbb{S}_r$.  We present a general estimate for sets $Z$ satisfying the condition $\alpha(Z)\geq 2$. The assumption of  very-ampleness of line bundle $\mathbb{L}_r$ is significant, thus our result concerns surfaces $\mathbb{S}_r$ with 
$r \leq 6$.

\begin{theorem}
Let $1 \leq r \leq 6$  and $Z \subset \mathbb{S}_r$ be a finite set of points such that  $\alpha (Z)\geq 2$. Then we have
$$\frac{\alpha (mZ)}{m}\geq \frac{\alpha (Z)-1}{2}.$$

\end{theorem}

\begin{proof}

We  recall first the notation $\mathbb{L}_r =3H- E_1- \dots - E_r $. We assume that $\alpha (Z)\geq 2$. Let us denote $\alpha (Z)$ by $\alpha$. We have
$$h^0 (3mH - mE_1 - \dots -m E_r)= \binom{3m+2}{2}- r \cdot \binom{m+1}{2}= \frac{(9-r)m^2 +(9-r)m +2}{2}.$$
We choose a minimal subset $W\subseteq Z$, such that $\alpha(W) = \alpha$, i.e., there is no  element in $(\alpha-1) \cdot \mathbb{L}_r$.
The minimality of $W$ is taken with respect to the inclusion (thus there can be several sets satisfying this condition).
It follows that the points in $W$ impose independent conditions on the space of sections in $|(\alpha-1) \cdot \mathbb{L}_r |$.
Then
$$\# W=t=\binom{3 \alpha -1}{2} - r \cdot \binom{\alpha}{2}=\frac{(9-r) {\alpha}^2 - (9-r)\alpha +2}{2}.$$
We claim that $|\alpha \cdot \mathbb{L}_r \otimes I_{W}|$ has no additional base points on $\mathbb{S}_r$, i.e. is not contained in $W$.
Let $W=\{Q_1, \dots, Q_{t} \}$. For any $Q_{i}$, there exists a curve $C_{i} \in (\alpha -1) \cdot \mathbb{L}_r$,
such that $C_{i}$ does not vanish at $Q_{i}$ and it does vanish at all points in $W \setminus \{Q_{i}\}$. Let
$s_{i}$ denote the section in $H^0 (\mathbb{S}_r, (\alpha -1) \cdot \mathbb{L}_r)$ corresponding to $C_i$. Then
the sections $s_1, \dots, s_t$ form a basis of $H^0 (\mathbb{S}_r, (\alpha -1) \cdot \mathbb{L}_r)$.

Suppose that $R_i \in \mathbb{S}_r \setminus W$ is a base point of $|\alpha \cdot \mathbb{L}_r \otimes I_{W}|$. There exists a section
$s_i \in \{ s_1, \dots, s_t\}$ not vanishing at $R$. Indeed, otherwise $R$ would be a common zero of $|(\alpha -1) \cdot \mathbb{L}_r)|$,
which is not possible by the choice of $W$. Since $\mathbb{L}_r$ is very ample, the system $| \mathbb{L}_r \otimes I_{Q_i}|$ is then
base point free away from $Q_i$. Hence there exists a section $s \in H^0 (\mathbb{S}_r,  \mathbb{L}_r \otimes I_{Q_i})$ not
vanishing at $R$. Then, in particular, $|\alpha \cdot \mathbb{L}_r \otimes I_{W}|$ has no base component. Thus
$$ s_{i} \cdot s \in H^0 (\mathbb{S}_r, (\alpha -1) \cdot \mathbb{L}_r \otimes I_{W\setminus \{Q\}} \otimes
\mathbb{L}_r  \otimes I_{Q_i})=H^0 (\mathbb{S}_r, \alpha  \cdot \mathbb{L}_r \otimes I_W)$$ is a section
not vanishing at $R$. Let $A \in |\alpha(Z) \mathbb{L}_r|$ and $B \in |\alpha(mZ) \mathbb{L}_r|$. Using Bezout theorem we obtain
$$\alpha \cdot \alpha(mZ) \cdot L_{r}^2= A \cdot B  \geq \frac{(9-r) {\alpha}^2 - (9-r)\alpha +2}{2} \cdot m$$
what finally implies
$$\frac{\alpha(mZ)}{m} \geq  \frac{(9-r) {\alpha}^2 - (9-r)\alpha +2}{2 \alpha (9-r)} > \frac{(9-r) {\alpha}^2 - (9-r)\alpha}{2 \alpha (9-r)}> \frac{\alpha -1}{2}.$$
This ends the proof.
\end{proof}

\begin{remark}
 Let us notice that if $\alpha(Z)=1$, then $\frac{\alpha(mZ)}{m} \geq \frac{1}{5}$. Firstly let us observe that if $Z$ is set from Theorem \ref{s1.5jumps}, then its initial sequence is of the form $$1,1,1,1,1,2,2,2,2,2,3,3,3,3,3,4, \dots$$
 The divisor $F=3k L$ for the line $L$ passing through the $P_1$ and corresponding to the point $Q \in E_1$, gives rise to $D=3k \widetilde{L} +2 k E_1 \in |3kH-kE_1|$
	on the blow up $\mathbb{S}_1$ and $\mult_{Q}(D)=5k$ for any $Q \in Z$. Hence $\alpha(5kZ)\leq k$ for any positive integer $k$.
	
	For $k=1$ we then obtain $\alpha(5Z)\leq 1$, what means that $\alpha(Z)= \dots = \alpha(5Z) =1$. Moreover by Remark \ref{bound}
	we conclude
	\begin{equation}
	\label{5wzor}
	\alpha(6Z) \geq 2.
	\end{equation}
	On the other hand, for $k=2$ we have
	
	\begin{equation}
	\label{5wzor1}
	\alpha(10Z)\leq 2.
	\end{equation}
	From (\ref{5wzor}) and (\ref{5wzor1}) we then obtain $\alpha(6Z)= \dots = \alpha(10Z)=2$.
	
	Using the same argumentation for the next $k$ we finally conclude, that the initial sequence in this case is $\alpha{(mZ)}=  \lceil \frac{m}{5} \rceil$ and indeed $\frac{\alpha(mZ)}{m} \geq \frac{1}{5}$. 
	
	Let $\{\alpha'(mZ)\}$ be an another subadditive and weakly growing sequence of positive integers with $\alpha'(Z)=1$. By Remark \ref{bound}
	$$\alpha'(mZ) \geq \alpha(mZ)$$ for any $m$, thus
	$$\frac{\alpha'(mZ)}{m} \geq \frac{\alpha(mZ)}{m} \geq \frac{1}{5}.$$
	By Lemma \ref{systems} we conclude, that estimate $\frac{\alpha(mZ)}{m} \geq \frac{1}{5}$ concerns any initial sequence $\{\alpha(mZ)\}$ with $\alpha(Z)=1$ for all surfaces $\mathbb{S}_r$. In the case of surfaces $\mathbb{S}_1$ and $\mathbb{S}_2$ we were able to show that this estimate is optimal (in the sense that $\frac{1}{5}$ is borderline value). Probably this estimate is not sharp for $r \geq 3$.
 \end{remark}

\section{Surface $\mathbb{S}_1$ as a del Pezzo surface and as a Hirzebruch surface}

The surface $\mathbb{S}_1$ was considered with respect to the fattening effect in \cite{DLS13} as a Hirzebruch surface. An interesting phenomenon is that
from the point of view of Hirzebruch surfaces the most natural choice of the reference line bundle for  $\mathbb{S}_1$ is
$$2H- E_1,$$
while if we consider it as a del Pezzo surface, we work with the anticanonical line bundle, i.e.,
$$\mathbb{L}_1 =3H- E_1.$$
Di Rocco, Lundman, and Szemberg
proved in \cite{DLS13} that on the Hirzebruch surface $\mathbb{S}_1$ (denoted there by $\mathbb{F}_1$) with
$2H- E_1$ there does not exist any finite set $Z$, such that
$$\alpha(Z)= \alpha(2Z)=\alpha(3Z)= \alpha(4Z)$$
(see \cite{DLS13}, Proposition $4.1$). From point of view of del Pezzo surfaces with the bundle $\mathbb{L}_1$ we can even get
$$\alpha(Z)= \alpha(2Z)=\alpha(3Z)= \alpha(4Z)=\alpha(5Z),$$
and moreover there exist infinitely many sets satisfying it (all singletons $Z=\{Q\}$ with $Q \in E_1$). Thus the choice of line bundle is a fundamental factor affecting the shape of the initial sequence.

\paragraph*{Acknowledgements.}

 I would like to thank professor Tomasz Szemberg for his great support throughout my PhD studies. I am also grateful to Piotr Pokora for helpful remarks on this text.

   The research of the author was partially supported by National Science Centre, Poland, grant 2016/23/N/ST1/01363.


\bigskip \footnotesize

\bigskip
Magdalena~Lampa-Baczy\'nska,
   Instytut Matematyki UP,
   Podchor\c a\.zych 2,
   PL-30-084 Krak\'ow, Poland
\\
\nopagebreak

  \textit{E-mail address:} \texttt{lampa.baczynska@wp.pl}


\end{document}